\documentclass[12pt]{amsart}

\usepackage{amssymb}

\makeatletter
\@namedef{subjclassname@2020}{%
  \textup{2020} Mathematics Subject Classification}
\def\author@andify{\nxandlist{\unskip,\penalty-1 \space\ignorespaces}{\unskip,\penalty-1 \space\ignorespaces}{\unskip,\penalty-2 \space\ignorespaces}}
\makeatother

\newtheorem{theorem}{Theorem}[section]
\newtheorem{corollary}[theorem]{Corollary}
\newtheorem{lemma}[theorem]{Lemma}
\newtheorem{proposition}[theorem]{Proposition}

\theoremstyle{definition}
\newtheorem{remark}[theorem]{Remark}

\numberwithin{equation}{section}

\usepackage[letterpaper,textwidth=13.5cm,textheight=23cm,hcentering,vcentering]{geometry}
\newcommand{\Z}{\mathbb{Z}}
\newcommand{\N}{\mathbb{N}}
\newcommand{\R}{\mathbb{R}}
\newcommand{\C}{\mathbb{C}}
\newcommand{\E}{\mathbb{E}}
\newcommand{\PP}{\mathbb{P}}
\newcommand{\Nodd}{N_{\mathrm{odd}}}
\DeclareMathOperator{\Mean}{M}

\begin{document}

\title[Chowla's cosine problem]{A half-constant upper bound for Chowla's cosine problem}

\author[B. Bu et al.]{Bryan Bu}
\address[Bryan Bu]{University of Massachusetts Amherst, 141 Commonwealth Avenue, Amherst, MA 01003, USA}
\email{bryanbu@umass.edu}

\author{Fayzan Khan}
\address[Fayzan Khan]{Summit High School, Summit, NJ, USA}
\email{itsfayzankhan@gmail.com}

\author{Hyunjun Yi}
\address[Hyunjun Yi]{Seoul International School, 15 Seongnam-daero 1518beon-gil, Sujeong-gu, Seongnam-si, Gyeonggi-do 13113, Republic of Korea}
\email{hjyi2027@gmail.com}

\author{Jeremy Barreto}
\address[Jeremy Barreto]{Lawrence High School, 2525 Princeton Pike, Lawrence\-ville, NJ 08648, USA}

\author{Atharv Jha}
\address[Atharv Jha]{Dougherty Valley High School, 10550 Albion Road, San Ramon, CA 94582, USA}
\email{jha.atharv27@gmail.com}

\author{Ishaan Chandupatla}
\address[Ishaan Chandupatla]{Old Bridge High School, 4209 Route 516, Matawan, NJ 07747, USA}
\email{Srmatiis@yahoo.com}

\author{Madhav Nagpal}
\address[Madhav Nagpal]{Monta Vista High School, 21840 McClellan Road, Cupertino, CA 95014, USA}
\email{madhavnagpal2806@gmail.com}

\thanks{Authors are listed in three tiers of contribution, alphabetically by surname within each tier, and the authors within each tier contributed equally: B.~Bu, F.~Khan and H.~Yi; J.~Barreto and A.~Jha; I.~Chandupatla and M.~Nagpal. Corresponding author: B.~Bu.}

\date{}

\begin{abstract}
Let $\lambda(n)$ be the infimum of $-\min_x\sum_{a\in A}\cos(ax)$ over all sets $A$ of $n$ positive integers. The classical Sidon-difference construction gives $\lambda(n)\le(1+o(1))\sqrt{n/2}$. We construct sets for which the negative excursion of the cosine sum has leading constant $1/2$: $\lambda(n)\le\frac12\sqrt n+O(n^{1/4}\sqrt{\log n})$, so that $\limsup_{n\to\infty}\lambda(n)/\sqrt n\le1/2$. The construction combines a two-sided Sidon square identity with a Singer difference set, a Fej\'er-kernel majorant, a parity-biased deletion, and an exact-cardinality probabilistic selection. We also record an $L^4$ criterion and a multiplicity bound showing what any construction with a smaller leading constant must achieve. In the opposite direction we carry the $L^1$ norm through Bedert's boundary estimate, obtaining it with an explicit constant and with one factor of $-\min f_A$ replaced by $\|f_A\|_1$, where $f_A(x)=\sum_{a\in A}\cos(ax)$, and deduce two lower bounds: $-\min f_A\ge cn^{1/4}$ for every $n$-element set $A\subset[1,Cn]$, and $(-\min f_A)^3\|f_A\|_1^2\ge cn$ for every $n$-element $A\subset\N$. The second recovers the exponent $1/5$ of Bedert and Shankar and improves it whenever $\|f_A\|_1$ is small compared with $-\min f_A$. Numerical values of the construction are reported up to $|A|=25440$.
\end{abstract}



\maketitle
\markleft{B. BU ET AL.}

\section{Introduction}

For a finite set $A\subset\N$ write
\[
f_A(x)=\sum_{a\in A}\cos(ax),
\]
and define
\[
\lambda(n)=\inf_{\substack{A\subset\N\\ |A|=n}}\Bigl(-\min_{x\in\R}f_A(x)\Bigr),
\]
the least possible depth of the negative excursion of a cosine sum with $n$ frequencies. Since $f_A$ has mean zero and $f_A(0)=n$, one has $\lambda(n)>0$. Ankeny and Chowla \cite{Cho52} asked whether $\lambda(n)\to\infty$; Uchiyama and Uchiyama \cite{UU60} proved this using Cohen's work on the Littlewood $L^1$ problem, and the resolution of that problem by Konyagin \cite{Kon81} and McGehee, Pigno and Smith \cite{MPS81} gives $\lambda(n)\gg\log n$. Bourgain \cite{Bou86} and Ruzsa \cite{Ruz04} obtained $\lambda(n)\ge\exp(c\sqrt{\log n})$, and polynomial lower bounds were proved recently: $\lambda(n)\ge n^{1/10-o(1)}$ by Jin, Milojevi\'c, Tomon and Zhang \cite{JMTZ25} and $\lambda(n)\ge n^{1/5-o(1)}$ by Bedert \cite{Bed25}, sharpened to $\lambda(n)\gg n^{1/5}$ by Shankar \cite{Sha26}; the spectral method of \cite{JMTZ25} has since been refined by Yip \cite{Yip26} to $\lambda(n)\gg n^{1/6.43}$. Chowla \cite{Cho65} conjectured that $\lambda(n)\asymp\sqrt n$. Explicit values are known only for $n=2,3$ \cite{Mer12}.

In the other direction, the standard construction, and the only one we are aware of, is the Sidon-difference construction. A set $S\subset\Z$ is \emph{Sidon} if every nonzero difference has at most one ordered representation, equivalently if $s_1-s_2=s_3-s_4$ and $s_1\ne s_2$ force $(s_1,s_2)=(s_3,s_4)$. If $B$ is a Sidon set of size $k$ and $A=\{b-b'>0: b,b'\in B\}$, then $|A|=\binom k2$ and
\[
f_A(x)=\tfrac12\Bigl(\Bigl|\sum_{b\in B}e^{ibx}\Bigr|^2-k\Bigr)\ge-\frac k2 ,
\]
so that $\lambda(\binom k2)\le k/2$, that is,
\begin{equation}\label{eq:folklore}
\lambda(m)\le(1+o(1))\sqrt{m/2}\qquad(m=\tbinom k2\to\infty).
\end{equation}
This gives the leading constant $1/\sqrt2\approx0.707$ along the cardinalities $\binom k2$. The explicit form recorded by Mercer \cite[Fact~2.1]{Mer19}, valid for every $m$, is $\lambda(m)\le\sqrt{2m}+1/2$; the larger constant there comes from padding the difference set with up to $\sqrt{2m}$ arbitrary frequencies, each bounded trivially. Recent papers quote the construction only as $\lambda(n)\ll\sqrt n$ with an unspecified or generous constant (for instance $10\sqrt n$ in \cite{Bed25}). We are not aware of any published construction with leading constant below $1/\sqrt2$.

In this paper we sharpen the leading constant to $1/2$. The idea is to use sums as well as differences of a Sidon set: for a real cosine sum $\rho$ with $k$ frequencies, $\rho^2$ produces $k(k-1)$ distinct nondiagonal frequencies while the diagonal contribution is only $k/2$. The price is a correction term (the sum $g$ below) that is not controlled by positivity alone, and the bulk of the paper is the construction of a Sidon set for which this correction is negligible at every point of the circle.

\subsection*{Results}
For $S=\{s_1,\dots,s_k\}\subset\Z$ define
\[
A(S)=\{3|s_i-s_j|: i<j\}\cup\{|3(s_i+s_j)+1|: i<j\},
\]
\[
\rho(x)=\sum_{j=1}^k\cos\bigl((3s_j+\tfrac12)x\bigr),\qquad
g(x)=\sum_{j=1}^k\cos\bigl((6s_j+1)x\bigr).
\]

\begin{proposition}[Two-sided Sidon identity]\label{prop:identity-intro}
If $S\subset\Z$ is Sidon and $|S|=k\ge2$, then $A(S)\subset\N$ has $|A(S)|=k(k-1)$, and for every real $x$,
\[
f_{A(S)}(x)=\rho(x)^2-\frac k2-\frac12 g(x).
\]
\end{proposition}

\begin{theorem}[Base cardinalities]\label{thm:base-intro}
For every sufficiently large integer $k$ there exists a set $A\subset\N$ with $|A|=k(k-1)$ such that
\[
\min_x f_A(x)\ge-\frac k2-O\bigl(\sqrt{k\log k}\bigr).
\]
Equivalently, $\lambda(k(k-1))\le k/2+O(\sqrt{k\log k})$.
\end{theorem}

\begin{theorem}[Half-constant upper bound]\label{thm:main-intro}
As $n\to\infty$,
\[
\lambda(n)\le\frac12\sqrt n+O\bigl(n^{1/4}\sqrt{\log n}\bigr),
\qquad\text{so that}\qquad
\limsup_{n\to\infty}\frac{\lambda(n)}{\sqrt n}\le\frac12 .
\]
\end{theorem}

Section~\ref{sec:limits} explains why $1/2$ is the natural endpoint of this method: an $L^4$ identity (Proposition~\ref{prop:L4}) and a multiplicity bound (Proposition~\ref{prop:multiplicity}) show that in any construction of the form ``square of a trigonometric polynomial with coefficients of constant modulus, minus an analytically absorbed correction'', a leading constant below $1/2$ requires every frequency to have at least three representations as a difference of the support, whereas the only symmetry of $\Z$ produces two. Section~\ref{sec:dense} concerns the opposite direction. The boundary estimate of Bedert \cite[Proposition~7.3]{Bed25} bounds, for a symmetric set $A'\subset\Z\setminus\{0\}$ whose exponential sum is at least $-K$, the number of maximal arithmetic progressions in $A'$ with a given common difference and at least two terms by $O(K^4)$. We reproduce his argument with the $L^1$ norm of the exponential sum carried through and obtain the bound $O(K^3\|\widehat{\mathbf 1_{A'}}\|_1)$ with an explicit constant (Proposition~\ref{prop:boundary}). Combined with an elementary count of three-term progressions this gives $-\min f_A\ge cn^{1/4}$ for every $A\subset[1,Cn]$ with $|A|=n$ (Proposition~\ref{prop:dense}), a class that contains the sets of Theorem~\ref{thm:base-intro}, and for which the finite-group results of Sanders \cite{San10} give a polynomial bound with an exponent that decays with $C$. Combined with Shankar's total-boundary identity \cite{Sha26} and a third-moment inequality for the positive part of $f_A$ it gives $(-\min f_A)^3\|f_A\|_1^2\ge cn$ for every $n$-element $A\subset\N$ (Theorem~\ref{thm:L1}), which contains the bound $-\min f_A\gg n^{1/5}$ of \cite{Bed25,Sha26} because $\|f_A\|_1\le-2\min f_A$, and improves it when the $L^1$ norm is smaller; a companion form of Ruzsa's lemma bounds arithmetic progressions in $A$ by $64(-\min f_A)\|f_A\|_1+6$ (Corollary~\ref{cor:APL1}). Section~\ref{sec:numerics} reports numerical values of the construction.

\subsection*{Organisation}
Section~\ref{sec:identity} proves Proposition~\ref{prop:identity-intro}. Section~\ref{sec:singer} sets up the Singer difference set and the Fej\'er-kernel majorant $m(y)$. Section~\ref{sec:parity} isolates a parity block of elements near the centre of the difference set. Section~\ref{sec:thinning} carries out the exact-cardinality probabilistic thinning. Section~\ref{sec:zones} performs the three-zone estimate and proves Theorem~\ref{thm:base-intro}. Section~\ref{sec:passage} passes from the base cardinalities $k(k-1)$ to every cardinality $n$ and proves Theorem~\ref{thm:main-intro}. Section~\ref{sec:limits} contains the $L^4$ criterion, Section~\ref{sec:dense} the lower bounds, and Section~\ref{sec:numerics} the numerical results.

\section{The square identity}\label{sec:identity}

\begin{proposition}[Two-sided Sidon identity]\label{prop:identity}
If $S\subset\Z$ is Sidon and $|S|=k\ge2$, then $A(S)\subset\N$ has
\[
|A(S)|=k(k-1),
\]
and for every real $x$,
\[
f_{A(S)}(x)=\rho(x)^2-\frac k2-\frac12g(x).
\]
\end{proposition}

\begin{proof}
The frequencies $3|s_i-s_j|$ are divisible by $3$, whereas $|3(s_i+s_j)+1|$ is congruent to $\pm1$ modulo $3$, so the two families are disjoint, and all frequencies are positive. The Sidon property gives distinct difference frequencies. If
\[
|3(s_i+s_j)+1|=|3(s_a+s_b)+1|,
\]
the negative-sign alternative would imply $3(s_i+s_j+s_a+s_b)=-2$, which is impossible. Thus $s_i+s_j=s_a+s_b$, and the Sidon property implies $\{i,j\}=\{a,b\}$. Hence each family has $\binom k2$ elements.

Writing $\alpha_j=3s_j+\frac12$ and expanding,
\begin{align*}
\rho(x)^2&=\sum_j\cos^2(\alpha_jx)+2\sum_{i<j}\cos(\alpha_ix)\cos(\alpha_jx)\\
&=\frac k2+\frac12\sum_j\cos\bigl((6s_j+1)x\bigr)+\sum_{i<j}\cos\bigl(3(s_i-s_j)x\bigr)\\
&\quad+\sum_{i<j}\cos\bigl((3(s_i+s_j)+1)x\bigr).
\end{align*}
Since cosine is even and all frequencies above occur exactly once, the last two sums are $f_{A(S)}(x)$. Rearranging proves the identity.
\end{proof}

The significance of Proposition~\ref{prop:identity} is that $k(k-1)\sim k^2$ positive frequencies are created while the unavoidable diagonal contribution is only $k/2$. The remainder of the proof constructs $S$ so that $\rho^2-g/2\ge-o(k)$ uniformly in $x$. Note that positivity of $\rho^2$ alone does not suffice: at a zero $x_0$ of $\rho$ one has $f_{A(S)}(x_0)=-k/2-g(x_0)/2$, and $g$ is a cosine sum with $k$ terms.

\section{Singer sets and Fej\'er smoothing}\label{sec:singer}

Let $q$ be a prime and
\[
v=q^2+q+1 .
\]
We use a Singer difference set \cite{Sin38} $D\subset\Z/v\Z$ of size $q+1$ with the perfect-difference property: every nonzero residue modulo $v$ has a unique representation $d-d'$ with $d,d'\in D$, $d\ne d'$. Choose representatives $D\subset\{0,\dots,v-1\}$ and put
\[
c=\frac{v-1}2,\qquad L=c+1=\frac{v+1}2 .
\]
Every subset of $D-c\subset\Z$ is Sidon over the integers, since a coincidence of differences over $\Z$ would be a coincidence modulo $v$.

Define the tent weight
\[
w(n)=1-\frac{|n-c|}{L},\qquad 0\le n\le v-1,
\]
and the Fej\'er kernel
\[
F_M(y)=\frac1M\Bigl(\frac{\sin(My/2)}{\sin(y/2)}\Bigr)^2 .
\]
Set
\[
m(y)=\frac{q+1}{v}F_L(y),\qquad \kappa=m(0)=\frac{(q+1)L}{v}.
\]
Thus
\[
\frac{q+1}2\le\kappa\le\frac{q+2}2 .
\]
We shall use
\begin{equation}\label{eq:mbounds}
0\le m(y)\le\kappa,\qquad m(y)\le\frac{q+1}{vL}\cdot\frac1{\sin^2(y/2)},
\end{equation}
and, for $|u|<\pi$,
\begin{equation}\label{eq:doubling}
m(2u)\le\frac{m(u)}{\cos^2(u/2)} .
\end{equation}
The last inequality follows from $\sin^2(Lu)=4\sin^2(Lu/2)\cos^2(Lu/2)\le4\sin^2(Lu/2)$ and $\sin^2u=4\sin^2(u/2)\cos^2(u/2)$.

\begin{lemma}[Singer flat spectrum]\label{lem:flat}
For $j\not\equiv0\pmod v$,
\[
\Bigl|\sum_{d\in D}e^{-2\pi ijd/v}\Bigr|^2=q .
\]
Consequently, for every $\phi:\{0,\dots,v-1\}\to\C$, if
\[
\Phi(j)=\sum_{n=0}^{v-1}\phi(n)e^{2\pi ijn/v},
\]
then
\[
\Bigl|\sum_{d\in D}\phi(d)-\frac{q+1}v\Phi(0)\Bigr|\le\frac{\sqrt q}{v}\sum_{j=1}^{v-1}|\Phi(j)| .
\]
\end{lemma}

\begin{proof}
The perfect-difference property gives
\[
\Bigl|\sum_{d\in D}e^{-2\pi ijd/v}\Bigr|^2=q+1+\sum_{u=1}^{v-1}e^{2\pi iju/v}=q .
\]
By Fourier inversion $\phi(n)=v^{-1}\sum_j\Phi(j)e^{-2\pi ijn/v}$, so
\[
\sum_{d\in D}\phi(d)=\frac1v\sum_{j=0}^{v-1}\Phi(j)\sum_{d\in D}e^{-2\pi ijd/v},
\]
and the second assertion follows from the first and the triangle inequality.
\end{proof}

\begin{lemma}[Tent-weighted Singer sum]\label{lem:tent}
For every real $y$,
\[
\sum_{d\in D}w(d)e^{i(d-c)y}=m(y)+R(y),\qquad|R(y)|\le\sqrt q .
\]
\end{lemma}

\begin{proof}
The Fourier series of the tent is
\[
\sum_{n=0}^{v-1}w(n)e^{iny}=e^{icy}F_L(y),
\]
since $n-c$ ranges over $|n-c|\le L-1$ and $\sum_{|m|<L}(1-|m|/L)e^{imy}=F_L(y)$. Apply Lemma~\ref{lem:flat} with $\phi(n)=w(n)e^{i(n-c)y}$. Then $\Phi(j)=e^{2\pi ijc/v}F_L(y+2\pi j/v)$, so $\frac{q+1}v\Phi(0)=m(y)$ and $|\Phi(j)|=F_L(y+2\pi j/v)\ge0$. The discrete aliasing identity
\[
\sum_{j=0}^{v-1}F_L(y+2\pi j/v)=v,
\]
valid because no nonzero $|m|<L$ is divisible by $v$, shows that $\sum_{j\ne0}|\Phi(j)|\le v$, and the total nonzero-frequency error is at most $\sqrt q$.
\end{proof}

\section{A parity block near the centre}\label{sec:parity}

For $0\le r\le c$ let
\[
\Nodd(r)=\#\{d\in D:|d-c|\le r,\ d-c\text{ odd}\}.
\]

\begin{lemma}[Odd-offset counting]\label{lem:odd}
Let $1\le h\le a$ and $a+h\le c$. Then
\[
\Nodd(a+h)\ge\frac{q+1}v\Bigl(a-\frac1{4h}\Bigr)-\frac ah\sqrt q .
\]
\end{lemma}

\begin{proof}
Define the trapezoidal cutoff
\[
\psi(u)=\frac{\max(0,a+h-|u|)-\max(0,a-h-|u|)}{2h}.
\]
Then $0\le\psi\le1$, it is supported on $|u|<a+h$, it equals $1$ for $|u|\le a-h$, and
\[
\widehat\psi(\theta):=\sum_{u\in\Z}\psi(u)e^{iu\theta}=\frac{(a+h)F_{a+h}(\theta)-(a-h)F_{a-h}(\theta)}{2h}.
\]
In particular,
\[
\widehat\psi(0)=2a,\qquad|\widehat\psi(\pi)|\le\frac1{2h},
\]
the latter because $MF_M(\pi)\in\{0,1\}$. Put
\[
\phi(n)=\psi(n-c)\cdot\frac{1-(-1)^{n-c}}2 .
\]
Then $\phi$ is bounded above by the indicator counted by $\Nodd(a+h)$, while
\[
\Phi(0)=\frac{\widehat\psi(0)-\widehat\psi(\pi)}2\ge a-\frac1{4h}.
\]
Using the nonnegative majorant
\[
\widetilde\psi(\theta)=\frac{(a+h)F_{a+h}(\theta)+(a-h)F_{a-h}(\theta)}{2h}
\]
we have $|\Phi(j)|\le\frac12\bigl(\widetilde\psi(2\pi j/v)+\widetilde\psi(2\pi j/v+\pi)\bigr)$, and Fej\'er aliasing gives
\[
\sum_{j=1}^{v-1}|\Phi(j)|\le\frac{av}h .
\]
Lemma~\ref{lem:flat} now yields the claimed bound.
\end{proof}

\section{Exact-cardinality thinning}\label{sec:thinning}

Fix a large integer $k$. Choose a prime $q$ with
\[
4k<q\le8k,
\]
which exists by Bertrand's postulate. Define
\[
N=v(q+1),\qquad
\Delta=\sqrt{(q+1)\log\bigl(64N\sqrt{q+1}\bigr)}+\frac\pi2,\qquad
\Delta'=\Delta+\sqrt q,
\]
\[
r=\lceil15\Delta'\rceil,\qquad\Lambda=\Delta'+r,
\]
and
\[
h_0=\Bigl\lceil\frac v{5(q+1)}\Bigl(r+5\sqrt q+\frac14\Bigr)\Bigr\rceil,\qquad
\rho_0=6h_0,\qquad
\tau_0=\frac15\sqrt{\frac{q+1}{vL\Lambda}} .
\]
For all sufficiently large $q$ these parameters satisfy
\begin{equation}\label{eq:params1}
\rho_0\le\frac{3L}{10},\qquad3\rho_0\tau_0\le\frac25,\qquad\tau_0\le0.05,
\end{equation}
and
\begin{equation}\label{eq:params2}
\Delta'=O\bigl(\sqrt{q\log q}\bigr),\qquad\Lambda=O\bigl(\sqrt{q\log q}\bigr),\qquad
0.09\Delta'\ge\sqrt{0.15(q+2)} .
\end{equation}
Indeed $N\asymp q^3$ gives $\Delta\asymp\sqrt{q\log q}$, hence $r\asymp\Lambda\asymp\sqrt{q\log q}$ and $h_0\asymp q^{3/2}\sqrt{\log q}$, while $L\asymp q^2$; since $(q+1)/(vL)\asymp q^{-3}$ we get $\tau_0\asymp q^{-3/2}\Lambda^{-1/2}\asymp q^{-7/4}(\log q)^{-1/4}$, so $\rho_0\tau_0\asymp q^{-1/4}(\log q)^{1/4}\to0$. The last inequality in \eqref{eq:params2} holds as soon as $\log q\ge7$, because $\log(64N\sqrt{q+1})\ge3.5\log q$ gives $\Delta'\ge\sqrt{3.5(q+1)\log q}$.

By Lemma~\ref{lem:odd} with $a=5h_0$ and $h=h_0$, there are at least
\[
\frac{q+1}v\Bigl(5h_0-\frac1{4h_0}\Bigr)-5\sqrt q\ \ge\ r
\]
elements $d\in D$ satisfying
\[
|d-c|\le\rho_0,\qquad d-c\text{ odd}.
\]
Fix a set $P$ of exactly $r$ such elements. Since $\rho_0\le3L/10$,
\[
w(d)\ge0.7\qquad(d\in P).
\]
Put
\[
Q_P=\sum_{d\in P}w(d),\qquad M=\sum_{d\in D\setminus P}w(d)=\kappa+R(0)-Q_P,\qquad\theta=\frac kM ,
\]
where $|R(0)|\le\sqrt q$ by Lemma~\ref{lem:tent} at $y=0$. Since $0\le Q_P\le r$, we have $\kappa-\sqrt q-r\le M\le\kappa+\sqrt q$, and for sufficiently large $k$ the inequalities $4k<q\le8k$, $(q+1)/2\le\kappa\le(q+2)/2$ and $r=o(k)$ imply
\begin{equation}\label{eq:theta}
\frac15\le\theta\le1 .
\end{equation}

For $d\in D\setminus P$, independently choose $\xi_d\in\{0,1\}$ with
\[
\PP(\xi_d=1)=p_d=\theta w(d),
\]
and put $\xi_d=0$ on $P$. Define
\[
S=\{d-c:\xi_d=1\},\qquad T(y)=\sum_{s\in S}e^{isy}.
\]
Then
\[
\E|S|=\sum_{d\notin P}\theta w(d)=k .
\]

\begin{lemma}[Uniform noise bound]\label{lem:noise}
Let
\[
\mathcal D(y)=T(y)-\E T(y).
\]
Then
\[
\PP\Bigl(\sup_{y\in\R}|\mathcal D(y)|>\Delta\Bigr)\le\frac1{16\sqrt{q+1}} .
\]
\end{lemma}

\begin{proof}
For fixed $y$, Hoeffding's inequality applied separately to the real and imaginary parts gives
\[
\PP\bigl(|\mathcal D(y)|\ge t\bigr)\le4\exp\Bigl(-\frac{t^2}{q+1}\Bigr).
\]
Take the grid $y_\ell=2\pi\ell/N$, $0\le\ell<N$, and
\[
t_0=\sqrt{(q+1)\log\bigl(64N\sqrt{q+1}\bigr)} .
\]
A union bound gives grid failure probability at most $1/(16\sqrt{q+1})$. Since $\mathcal D$ has degree at most $c$ and $\|\mathcal D\|_\infty\le q+1$, Bernstein's inequality gives $\|\mathcal D'\|_\infty\le c(q+1)$. Every $y$ lies within $\pi/N$ of a grid point, hence
\[
|\mathcal D(y)-\mathcal D(y_\ell)|\le\frac\pi Nc(q+1)=\frac{\pi c}v<\frac\pi2 .
\]
Thus the grid bound $t_0$ implies the uniform bound $\Delta=t_0+\pi/2$.
\end{proof}

\begin{lemma}[Exact cardinality]\label{lem:exact}
Let $X$ be a sum of $s$ independent Bernoulli variables and suppose $\E X=k\in\Z$. Then
\[
\PP(X=k)\ge\frac1{4\sqrt s} .
\]
\end{lemma}

\begin{proof}
By Darroch's theorem \cite{Dar64} on the mode of a Poisson-binomial distribution, the integer mean $k$ is a mode of $X$. Let $\sigma^2=\operatorname{Var}(X)\le s/4$. If $\sigma=0$, the assertion is immediate. Otherwise Chebyshev's inequality gives
\[
\PP(|X-k|<2\sigma)\ge\frac34 .
\]
The interval $(k-2\sigma,k+2\sigma)$ contains at most $4\sigma+1$ integers, so one of them has probability at least
\[
\frac3{4(4\sigma+1)}\ge\frac1{4\sqrt s}.
\]
Since $k$ is a mode, the same lower bound holds at $k$.
\end{proof}

Let $G$ be the event $\sup_y|\mathcal D(y)|\le\Delta$. Here $s\le q+1$, and therefore Lemmas~\ref{lem:noise} and~\ref{lem:exact} give
\[
\PP(G^c)\le\frac1{16\sqrt{q+1}}<\frac1{4\sqrt s}\le\PP(|S|=k).
\]
Hence there exists a realization satisfying simultaneously
\begin{equation}\label{eq:realization}
|S|=k,\qquad\sup_y|\mathcal D(y)|\le\Delta .
\end{equation}
Fix such a realization from now on.

\section{The three-zone estimate}\label{sec:zones}

By Lemma~\ref{lem:tent},
\[
\E T(y)=\theta\sum_{d\in D\setminus P}w(d)e^{i(d-c)y}=\theta\bigl(m(y)+R(y)-Q(y)\bigr),
\]
where
\[
Q(y)=\sum_{d\in P}w(d)e^{i(d-c)y}.
\]
Thus
\begin{equation}\label{eq:T-decomp}
T(y)=\theta m(y)+E(y),\qquad|E(y)|\le\Lambda,
\end{equation}
because $|E|\le\theta\sqrt q+\theta Q_P+\Delta\le\sqrt q+r+\Delta=\Lambda$. Since $d-c$ is odd on $P$, for every real $t$,
\begin{equation}\label{eq:T-shifted}
T(\pi+3t)=\theta\bigl(m(\pi+3t)+R(\pi+3t)+Q_0(t)\bigr)+\mathcal D(\pi+3t),
\end{equation}
where
\[
Q_0(t)=\sum_{d\in P}w(d)e^{3i(d-c)t}.
\]
Moreover, since $\rho(x)=\operatorname{Re}\bigl(e^{ix/2}T(3x)\bigr)$ and $g(x)=\operatorname{Re}\bigl(e^{ix}T(6x)\bigr)$,
\begin{equation}\label{eq:rho-g}
\rho(x)\ge\cos(x/2)\theta m(3x)-\Lambda,\qquad g(x)\le\cos(x)\theta m(6x)+\Lambda .
\end{equation}
Because $f_A$ is even and $2\pi$-periodic, it suffices to consider $x\in[0,\pi]$. Write $A=A(S)$ and $f_A=\rho^2-k/2-g/2$ (Proposition~\ref{prop:identity}).

\subsection*{Zone I: $0\le x\le\tau_0$}
Put $\mu=\theta m(3x)$. By \eqref{eq:doubling} and $\tau_0\le0.05$,
\[
\theta m(6x)\le1.024\mu,\qquad\cos(x/2)\ge0.998 .
\]
Therefore
\[
g(x)\le1.024\mu+\Lambda,\qquad\rho(x)\ge0.998\mu-\Lambda .
\]
If $\mu\le4\Lambda$, then $g(x)\le5.1\Lambda$, and hence
\[
f_A(x)\ge-\frac k2-2.55\Lambda .
\]
If $\mu>4\Lambda$, then
\[
\rho(x)\ge0.748\mu,
\]
so $\rho(x)^2\ge0.559\mu^2\ge0.637\mu\ge g(x)/2$ (as $\mu>4\Lambda\ge2$), and in particular
\[
f_A(x)\ge-\frac k2 .
\]

\subsection*{Zone II: $x=\pi/3+t$, $|t|\le\tau_0$}
From \eqref{eq:rho-g}, $\cos x\le0.55$ and $m\le\kappa\le(q+2)/2$,
\[
\frac{g(x)}2\le0.15(q+2)+\frac\Lambda2 .
\]
Set $\phi=\pi/6+t/2$. By \eqref{eq:T-shifted},
\begin{align*}
\rho(x)={}&\theta m(\pi+3t)\cos\phi+\operatorname{Re}\Bigl(e^{i\phi}\bigl(\theta R(\pi+3t)+\mathcal D(\pi+3t)\bigr)\Bigr)\\
&+\theta\sum_{d\in P}w(d)\cos\bigl(\phi+3(d-c)t\bigr).
\end{align*}
The first term is nonnegative and the second is at least $-\Delta'$. For $d\in P$, \eqref{eq:params1} gives
\[
|\phi+3(d-c)t|\le\frac\pi6+\frac1{20}+3\rho_0\tau_0\le\frac\pi6+\frac9{20}<0.974,
\]
so
\[
\cos\bigl(\phi+3(d-c)t\bigr)\ge0.56 .
\]
Using $w(d)\ge0.7$, $\theta\ge1/5$, and $r\ge15\Delta'$ gives
\[
\theta\sum_{d\in P}w(d)\cos\bigl(\phi+3(d-c)t\bigr)\ge0.2\cdot0.7\cdot0.56\cdot15\Delta'\ge1.17\Delta' .
\]
Consequently
\[
\rho(x)\ge0.17\Delta'\ge0.09\Delta'\ge\sqrt{0.15(q+2)} .
\]
Thus $\rho(x)^2$ absorbs the potentially positive part of $g(x)/2$, and
\[
f_A(x)\ge-\frac k2-\frac\Lambda2 .
\]

\subsection*{Zone III: the remainder of $[0,\pi]$}
If $x\ge\pi/2$, then $\cos x\le0$, so \eqref{eq:rho-g} gives $g(x)\le\Lambda$ and hence
\[
f_A(x)\ge-\frac k2-\frac\Lambda2 .
\]
If $x<\pi/2$ but $x$ lies outside Zones I and II, then $6x\in(6\tau_0,3\pi)$ has distance at least $\delta:=6\tau_0$ from $2\pi\Z$, so $3x$ has distance at least $\delta/2$ from $\pi\Z$ and
\[
|\sin3x|\ge\sin(\delta/2)\ge0.49\delta ,
\]
using $\delta\le0.3$. By \eqref{eq:mbounds} and the definition of $\tau_0$,
\[
m(6x)\le\frac{q+1}{vL\sin^2(3x)}\le\frac{q+1}{0.24vL\delta^2}=\frac{25}{8.64}\cdot\Lambda\le3\Lambda .
\]
Since $\theta\le1$,
\[
g(x)\le4\Lambda,
\]
and therefore
\[
f_A(x)\ge-\frac k2-2\Lambda .
\]

Combining the three zones gives
\begin{equation}\label{eq:combined}
\min_xf_A(x)\ge-\frac k2-2.55\Lambda=-\frac k2-O\bigl(\sqrt{k\log k}\bigr),
\end{equation}
since $q\asymp k$ and $\Lambda=O(\sqrt{q\log q})$.

\begin{theorem}[Base cardinalities]\label{thm:base}
For every sufficiently large integer $k$ there exists a set $A\subset\N$ with $|A|=k(k-1)$ such that
\[
\min_xf_A(x)\ge-\frac k2-O\bigl(\sqrt{k\log k}\bigr).
\]
Equivalently, $\lambda(k(k-1))\le k/2+O(\sqrt{k\log k})$.
\end{theorem}

This proves Theorem~\ref{thm:base-intro}.

\begin{remark}\label{rem:constants}
The constants in the definitions of Section~\ref{sec:thinning} are chosen for simplicity of verification, not for size: \eqref{eq:params1} forces $q$ to be of order at least $10^5$, and at such sizes $2.55\Lambda$ exceeds $150\sqrt{k\log k}$. The bound is therefore asymptotic in nature. The values actually attained by the construction at moderate sizes are far better; see Section~\ref{sec:numerics}.
\end{remark}

\section{Passage to every cardinality}\label{sec:passage}

\begin{lemma}[Subadditivity]\label{lem:subadd}
For nonnegative integers $a,b$,
\[
\lambda(a+b)\le\lambda(a)+\lambda(b).
\]
\end{lemma}

\begin{proof}
The cases $a=0$ or $b=0$ are immediate. For $a,b\ge1$, choose finite sets $A,B\subset\N$ of sizes $a,b$. Choose an integer $t>\max A$ sufficiently large that $A\cap tB=\emptyset$. Then
\[
f_{A\cup tB}(x)=f_A(x)+f_B(tx),
\]
so
\[
\min_xf_{A\cup tB}(x)\ge\min_xf_A(x)+\min_yf_B(y).
\]
Taking infima proves the claim.
\end{proof}

We use the explicit form of the Sidon-difference estimate \cite[Fact~2.1]{Mer19},
\begin{equation}\label{eq:mercer}
\lambda(m)\le\sqrt{2m}+\frac12\qquad(m\ge1).
\end{equation}

\begin{theorem}[Half-constant upper bound]\label{thm:main}
As $n\to\infty$,
\[
\lambda(n)\le\frac12\sqrt n+O\bigl(n^{1/4}\sqrt{\log n}\bigr).
\]
Consequently,
\[
\limsup_{n\to\infty}\frac{\lambda(n)}{\sqrt n}\le\frac12 .
\]
\end{theorem}

\begin{proof}
Let $k$ be maximal with $k(k-1)\le n$ and put
\[
m=n-k(k-1).
\]
Then $0\le m<2k$ and $k=\sqrt n+O(1)$. By Theorem~\ref{thm:base}, Lemma~\ref{lem:subadd}, and \eqref{eq:mercer},
\begin{align*}
\lambda(n)&\le\lambda(k(k-1))+\lambda(m)\\
&\le\frac k2+O\bigl(\sqrt{k\log k}\bigr)+O(\sqrt m)\\
&=\frac12\sqrt n+O\bigl(n^{1/4}\sqrt{\log n}\bigr).
\end{align*}
Dividing by $\sqrt n$ and letting $n\to\infty$ gives the limsup statement.
\end{proof}

This proves Theorem~\ref{thm:main-intro}.

\section{Limits of the method}\label{sec:limits}

All known constructions, including the one above, are of the form ``square of a trigonometric polynomial, minus the diagonal, minus a correction that is handled analytically''. We record an exact identity that shows what such a construction must achieve in order to have leading constant below $1/2$. For a trigonometric polynomial $F(x)=\sum_\gamma c_\gamma e^{i\gamma x}$ with real (not necessarily integer) frequencies $\gamma$, write $\Mean(F)=c_0$ for its mean; thus $\Mean(F)=\lim_{T\to\infty}\frac1{2T}\int_{-T}^TF$.

\begin{proposition}[$L^4$ criterion]\label{prop:L4}
Let $A\subset\N$ be finite with $|A|=n$, and suppose that for some $\lambda\ge0$, some trigonometric polynomial $P$ with complex coefficients and real frequencies, and some real trigonometric polynomial $h$ with $\Mean(h)=0$ whose frequency support is disjoint from $A\cup(-A)$,
\[
\lambda+f_A(x)=|P(x)|^2-h(x)\qquad(x\in\R).
\]
Then $\Mean(|P|^2)=\lambda$, $\Mean(|P|^4)=\lambda^2+\frac n2+\Mean(h^2)$, and
\[
\frac{\lambda}{\sqrt n}=\frac{\Mean(|P|^2)}{\sqrt{2\bigl(\Mean(|P|^4)-\Mean(|P|^2)^2-\Mean(h^2)\bigr)}} .
\]
In particular $\lambda<\frac12\sqrt n$ if and only if
\[
\Mean(|P|^4)-\Mean(h^2)>3\Mean(|P|^2)^2 .
\]
\end{proposition}

\begin{proof}
Taking means, $\Mean(f_A)=0$ and $\Mean(h)=0$ give $\Mean(|P|^2)=\lambda$. Squaring, $\Mean(|P|^4)=\Mean\bigl((\lambda+f_A+h)^2\bigr)=\lambda^2+\Mean(f_A^2)+\Mean(h^2)$, because $\Mean(f_A)=\Mean(h)=0$ and $\Mean(f_Ah)=0$ by the disjointness of supports. Finally $\Mean(f_A^2)=n/2$, since the $n$ frequencies of $f_A$ are distinct positive integers. The displayed formula and the equivalence follow.
\end{proof}

In Proposition~\ref{prop:identity} one has $P=\rho$, $\lambda=k/2$, $h=g/2$, $\Mean(g^2)=k/2$, and $\Mean(\rho^4)=\frac34k^2-\frac38k$, so the criterion reads $\frac34k^2-\frac12k>\frac34k^2$, which fails; this is consistent with $\lambda/\sqrt n=\frac12(1-1/k)^{-1/2}$.

For a finite $B\subset\R$ and $\gamma\in\R$ write $r_B(\gamma)=\#\{(\beta,\beta')\in B^2:\beta-\beta'=\gamma\}$. When $P$ has coefficients of constant modulus, the criterion takes a sharper, pointwise form.

\begin{proposition}[Multiplicity bound]\label{prop:multiplicity}
Let $A\subset\N$ be finite with $|A|=n$, and suppose that $\lambda+f_A=|P|^2-h$ with $h$ as in Proposition~\ref{prop:L4} and
\[
P(x)=c\sum_{\beta\in B}\varepsilon_\beta e^{i\beta x},\qquad c>0,\quad|\varepsilon_\beta|=1,
\]
for a finite set $B\subset\R$. Put $m=1/(2c^2)$. Then for every $\gamma\in A\cup(-A)$,
\[
\sum_{\beta-\beta'=\gamma}\varepsilon_\beta\overline{\varepsilon_{\beta'}}=m,\qquad\text{so in particular}\qquad r_B(\gamma)\ge m,
\]
and
\[
\lambda\ \ge\ \sqrt{\frac{n}{2m(1-1/|B|)}}\ >\ \sqrt{\frac n{2m}} .
\]
Consequently $\lambda<\frac12\sqrt n$ forces $m>2$: every frequency of $A$ then has at least three representations as a difference of elements of $B$.
\end{proposition}

\begin{proof}
Taking means in $\lambda+f_A=|P|^2-h$ gives $\lambda=\Mean(|P|^2)=c^2|B|$, because $\Mean(f_A)=\Mean(h)=0$. The coefficient of $e^{i\gamma x}$ in $|P|^2$ is $c^2\sum_{\beta-\beta'=\gamma}\varepsilon_\beta\overline{\varepsilon_{\beta'}}$. For $\gamma\in A\cup(-A)$ the coefficient of $e^{i\gamma x}$ in $\lambda+f_A$ is $\frac12$ and that in $h$ is $0$, so the sum equals $1/(2c^2)=m$, and $m\le r_B(\gamma)$ by the triangle inequality. Summing over the $2n$ elements of $A\cup(-A)$, which are distinct and nonzero,
\[
2nm\ \le\ \sum_{\gamma\in A\cup(-A)}r_B(\gamma)\ \le\ \sum_{\gamma\ne0}r_B(\gamma)=|B|^2-|B| .
\]
Since $|B|=\lambda/c^2=2m\lambda$, the right-hand side equals $4m^2\lambda^2(1-1/|B|)$, and the inequality reads $n\le2m\lambda^2(1-1/|B|)$. If $\lambda<\frac12\sqrt n$, then $n/4>\lambda^2\ge n/(2m(1-1/|B|))$ gives $m>2/(1-1/|B|)>2$, and $r_B(\gamma)\ge m>2$ for every $\gamma\in A\cup(-A)$.
\end{proof}

The proposition uses only the mean of $|P|^2$ and its coefficients on $A\cup(-A)$; it places no condition on $h$ beyond the support condition of Proposition~\ref{prop:L4}. For the identity of Proposition~\ref{prop:identity} one has $P=\rho=\frac12\sum_{\beta\in B}e^{i\beta x}$ with $B=\{\pm\alpha_j\}$, so $c=\frac12$, $m=2$, $|B|=2k$, and every frequency of $A(S)$ has $r_B=2$ exactly; the bound reads $\frac k2\ge\frac12\sqrt{k(k-1)\cdot2k/(2k-1)}$, which holds because $2(k-1)<2k-1$. The slack comes from the second inequality in the proof: the $2k$ differences $\pm2\alpha_j$ carry multiplicity one and lie outside $A(S)\cup(-A(S))$. The exact value is $\lambda/\sqrt n=\frac12(1-1/k)^{-1/2}$. The multiplicity two is supplied by the symmetry $\beta\mapsto-\beta$.

\begin{lemma}[Multiplicities of symmetric supports]\label{lem:involution}
Let $B\subset\R\setminus\{0\}$ be finite with $B=-B$. Then for every $\gamma\ne0$,
\[
r_B(\gamma)\equiv\mathbf 1[\gamma/2\in B]\pmod 2 .
\]
\end{lemma}

\begin{proof}
The map $(\beta,\beta')\mapsto(-\beta',-\beta)$ is an involution on the set of representations of $\gamma$. Its fixed points are the pairs $(\beta,-\beta)$, which exist if and only if $\gamma=2\beta$ with $\beta\in B$, and then there is exactly one.
\end{proof}

For $B=\{\pm\alpha_j\}$ with $\alpha_j=3s_j+\frac12$ and $S$ Sidon, Lemma~\ref{lem:involution} says that every frequency $\alpha_i\pm\alpha_j$ ($i\ne j$) has multiplicity exactly two while the frequencies $2\alpha_j=6s_j+1$ have multiplicity one; the latter are precisely the frequencies of the correction $g$ in Proposition~\ref{prop:identity}. Thus the identity of Section~\ref{sec:identity} is the extremal case of Lemma~\ref{lem:involution}.

\begin{remark}\label{rem:barrier}
Proposition~\ref{prop:multiplicity} bounds the constant $\lambda$ of the representation, whereas the quantity of interest is
\[
-\min_xf_A(x)=\lambda+\sup_x\bigl(h(x)-|P(x)|^2\bigr).
\]
Within the framework of Proposition~\ref{prop:L4} a leading constant below $\frac12$ can therefore arise in only two ways: either every frequency of $A$ has multiplicity $m>2$ in $B$ in the sense of Proposition~\ref{prop:multiplicity}, or $h\le|P|^2-\varepsilon\lambda$ on all of $\R$ for some fixed $\varepsilon>0$, that is, the correction is negative by a fixed proportion of $\lambda$ at every zero of $P$. We know of no support $B$ for which either can be carried out. For the first route, $\beta\mapsto-\beta$ is the only nontrivial automorphism of $\Z$ and supplies exactly multiplicity two (Lemma~\ref{lem:involution}), so larger multiplicities must come from additive coincidences in $B$; in the examples we have examined these are accompanied by a comparable number of frequencies whose coefficient in $|P|^2$ is not $1$, and the resulting $h$ is of order $\lambda$ rather than $o(\lambda)$. For a union of translates $B=C+D$ with $C$ Sidon, the multiplicities follow the profile of $r_D$, which is never constant on $D-D$; for the lift to $\Z$ of a cyclic difference set with parameter $\lambda_0\ge2$, the $\lambda_0$ representations of each nonzero residue are distributed over two integers. For the second route, in the construction of Section~\ref{sec:zones} the values of $g$ at the zeros of $\rho$ are essentially those of the random part $\mathcal D$ (Remark~\ref{rem:sharp}), and nothing forces them to be negative. Lowering the constant below $\frac12$ therefore appears to require an idea outside the square-identity paradigm.
\end{remark}

\begin{remark}[Sharpness for the present construction]\label{rem:sharp}
For the sets $A=A(S)$ of Theorem~\ref{thm:base}, at every zero $x_0$ of $\rho$ one has $f_A(x_0)=-k/2-g(x_0)/2$, so
\[
-\min_xf_A(x)\ \ge\ \frac k2+\frac12\sup\{g(x_0):\rho(x_0)=0\}.
\]
Thus the leading term $k/2$ could only be improved if $g$ were negative at every zero of $\rho$. Away from the neighbourhoods of $0$, $\pi/3$, $2\pi/3$ and $\pi$ both $\rho$ and $g$ consist of the random part $\mathcal D$ evaluated at $3x$ and $6x$ respectively, whose values at a fixed point are essentially uncorrelated, and $\rho$ has zeros on the scale $1/v$; one therefore expects the supremum above to be positive and of order $\sqrt{k\log k}$, the uniform size of $\mathcal D$, so that both the leading term and the order of the error term in Theorem~\ref{thm:base} should be sharp for this construction up to the logarithm. We do not prove this. The numerical minima in Section~\ref{sec:numerics} are attained at points where $\rho^2<1$, with $g/2$ between $1.05\sqrt k$ and $1.7\sqrt k$ (the penultimate column of Table~\ref{tab:num}).
\end{remark}

\section{Lower bounds: dense sets and the $L^1$ norm}\label{sec:dense}

The sets of Theorem~\ref{thm:base} have all their frequencies below $3v\le200n$ (Section~\ref{sec:thinning}: $S\subset[-c,c]$ with $c=(v-1)/2$, so every frequency $3|s_i-s_j|$ and $|3(s_i+s_j)+1|$ is at most $3(v-1)+1$; and $q\le8k$, $v=q^2+q+1$, $n=k(k-1)$, $k\ge29$ give $3v\le200n$). For sets of this kind a polynomial lower bound follows from Sanders' work on the analogous problem in finite abelian groups \cite{San10}, where for a symmetric $A\subset G$ one considers $M_G(A)=\sup_\gamma(-\widehat{\mathbf 1_A}(\gamma))$, the supremum over the characters of $G$. Sanders proves that a symmetric set $A\ni0$ of density $\alpha$ in $G$ either satisfies $|A\triangle H|\le K$ for some subgroup $H\le G$ or has $M_G(A)\ge K^{\alpha/5}/64$, for every $0<K\le|A|/2^{13}$ \cite[Theorem~4.1]{San10}, and that a symmetric $A\subset\Z/p\Z$ of size $(p+1)/2$ has $M_{\Z/p\Z}(A)\ge cp^{c'}$ with absolute constants $c,c'>0$, where he indicates that $c'=1/3$ can be arranged \cite[Theorem~1.2]{San10}. If $A\subset[1,Cn]$ has $n$ elements, then $A\cup(-A)\cup\{0\}$ embeds in $\Z/p\Z$, for a prime $p\in(4Cn,8Cn]$, as a symmetric set of density $\alpha>1/(4C)$ whose symmetric difference with each of the two subgroups of $\Z/p\Z$ has at least $2n$ elements, and the supremum over the characters of $\Z/p\Z$ is at most $-2\min f_A-1$; hence $-\min f_A\ge cn^{\alpha/5}$ with an absolute $c>0$, an exponent that tends to zero as $C$ grows. Proposition~\ref{prop:dense} below gives the exponent $1/4$ for every $C$, only the constant depending on $C$, with an elementary selection of the auxiliary shift. In the one case where Sanders' bound has an absolute exponent, sets of density exactly $1/2$ in $\Z/p\Z$, his exponent $1/3$ is the better one.

The tool is the boundary estimate of Bedert \cite[Proposition~7.3]{Bed25}. In Section~\ref{subsec:boundary} we reproduce its proof with the $L^1$ norm of the cosine sum carried through, which replaces one factor of $-\min f_A$ by $\|f_A\|_1$ and makes the constant explicit. Section~\ref{subsec:dense} treats sets with frequencies $O(n)$, and Section~\ref{subsec:L1} combines the estimate with Shankar's total-boundary identity \cite{Sha26} to obtain a lower bound for general sets in which two factors of $-\min f_A$ are replaced by $\|f_A\|_1$.

\subsection*{Notation}
Throughout this section $A'\subset\Z\setminus\{0\}$ is a finite symmetric set, $N=|A'|$, $a_m=\mathbf 1_{A'}(m)$, $e(y)=e^{2\pi iy}$, and
\[
F(y)=\widehat{\mathbf 1_{A'}}(y):=\sum_{a\in A'}e(ay),
\]
a real even trigonometric polynomial on $\mathbb T=\R/\Z$. More generally $\widehat{\mathbf 1_X}(y)=\sum_{x\in X}e(xy)$ for finite $X\subset\Z$. The Fourier coefficients of a function $g$ on $\mathbb T$ are $\widehat g(m)=\int_{\mathbb T}g(y)e(-my)\,dy$, so that $\widehat{\widehat{\mathbf 1_X}}=\mathbf 1_X$. All integrals are over $\mathbb T$ with respect to the probability Haar measure, $\|g\|_r=(\int|g|^r)^{1/r}$, and $(g*h)(y)=\int g(y-z)h(z)\,dz$, so that $\widehat{g*h}=\widehat g\widehat h$ and in particular $\widehat{\mathbf 1_X}*\widehat{\mathbf 1_Y}=\widehat{\mathbf 1_{X\cap Y}}$. We have $\int F=0$ and $\int F^2=N$. We fix $K>0$ with
\[
F+K\ge0\quad\text{on }\mathbb T,
\]
and write $u=F+K\ge0$, so that $\int u=K$ and $\widehat u(m)=a_m$ for $m\ne0$; hence $u*u=F+K^2=u+K^2-K$. We put $\mu=\|F\|_1$. Since $\int F=0$ we have $\int F_+=\int F_-=\mu/2$, where $F_\pm=\max(\pm F,0)$, and since $F_-\le K$,
\begin{equation}\label{eq:L1trivial}
\mu\le2K,\qquad \|F_-\|_2^2\le K\mu/2,\qquad \int F_-^{j}\le K^{j-1}\mu/2\quad(j\ge1).
\end{equation}
Moreover $K\ge1$ when $A'\ne\emptyset$: for $a_0\in A'$, $1=\int u(y)e(-a_0y)\,dy\le\int u=K$. For $t\in\Z\setminus\{0\}$ let
\[
B_t(A'):=\{a\in A': a-t\in A',\ a+t\notin A'\},
\]
the set of right endpoints of the maximal arithmetic progressions with common difference $t$ and at least two terms contained in $A'$.

For $A\subset\N$ finite we write $\lambda_A:=-\min_xf_A(x)$ and $\|f_A\|_1:=\frac1{2\pi}\int_0^{2\pi}|f_A|$. With $A'=A\cup(-A)$ one has $F(y)=2f_A(2\pi y)$, so the hypotheses above hold with $N=2n$, $K=2\lambda_A$ and $\mu=2\|f_A\|_1$, and \eqref{eq:L1trivial} reads $\|f_A\|_1\le2\lambda_A$.

\subsection{The boundary estimate}\label{subsec:boundary}

Bedert \cite[Proposition~7.3]{Bed25} proves that $|B_t(A')|\ll K^4$ for every $t\ne0$. His argument gives slightly more: one factor $K$ can be replaced by $\mu=\|F\|_1$, and the constant can be made explicit. Since this refinement is what Theorem~\ref{thm:L1} uses, we give the argument in full. Apart from the bookkeeping of $L^1$ norms it is Bedert's; the only changes are that $\|f_t\|_1\le4\mu$ replaces $\|f_t\|_1\le8K$ in the bound for the function $f_t*f_t$ below, and that the $L^1$ norms of the positive and negative parts of $F$ are kept in his decomposition lemma.

\begin{proposition}[Boundary estimate]\label{prop:boundary}
In the above setting, for every $t\in\Z\setminus\{0\}$,
\[
|B_t(A')|\ \le\ C_1K^3\mu\ \le\ 2C_1K^4,\qquad C_1=2\cdot10^5 .
\]
\end{proposition}

We need three lemmas. The first is \cite[Lemma~3.3]{Bed25}.

\begin{lemma}\label{lem:conv}
Let $g,h$ be real integrable functions on $\mathbb T$ with $\int g=\int h=0$, $g\ge-G$ and $h\ge-H$ for constants $G,H\ge0$. Then $g*h\ge-\frac12\bigl(G\|h\|_1+H\|g\|_1\bigr)$.
\end{lemma}

\begin{proof}
Write $g=g_+-g_-$ and $h=h_+-h_-$ with $0\le g_-\le G$, $0\le h_-\le H$, $\int g_+=\|g\|_1/2$ and $\int h_+=\|h\|_1/2$. Then $g*h\ge-(g_-*h_+)-(g_+*h_-)\ge-G\int h_+-H\int g_+$.
\end{proof}

For the second lemma write $T_1=F_+$ and $T_2=-F_-$, so that $F=T_1+T_2$, $0\le T_1\le u$, $-K\le T_2\le0$, $\|T_1\|_1=\|T_2\|_1=\mu/2$ and $\|T_2\|_2^2\le K\mu/2$, and put $e_s(y)=e(sy)$ for $s\in\Z$. Besides Parseval and Young's inequality $\|g*h\|_2\le\|g\|_1\|h\|_2$ we use the following splitting device: if $g$ is a function on $\mathbb T$ with $|g|\le u*u=u+K^2-K$, then $g=g'+g''$ with $g'=g\min(1,u/|g|)$ (so $g'=g$ where $|g|\le u$), and
\begin{equation}\label{eq:split}
|g'|\le u,\qquad |g''|=(|g|-u)_+\le K^2,\qquad |g''|\le|g|,
\end{equation}
hence $\|g''\|_2\le(\|g''\|_\infty\|g''\|_1)^{1/2}\le K\|g\|_1^{1/2}$.

\begin{lemma}[Decomposition]\label{lem:decomp}
For every $t\ne0$ there are functions $Q_1,Q_2$ on $\mathbb T$ such that
\[
\widehat{\mathbf 1_{B_t(A')}}=Q_1+Q_2,\qquad |Q_1|\le3u,\qquad Q_1(-y)=\overline{Q_1(y)},
\]
\[
\|Q_2\|_2\le6K\mu+\sqrt2\,K^{3/2}\mu^{1/2}\le7\sqrt2\,K^{3/2}\mu^{1/2}.
\]
\end{lemma}

\begin{proof}
Write $B=B_t(A')$ and $A'_t=A'\cap(A'+t)$, so that $A'_{-t}=-A'_t$, $B=A'_t\setminus A'_{-t}$ and $-B=A'_{-t}\setminus A'_t$. Since $(e_sF)^\wedge(m)=a_{m-s}$, comparing Fourier coefficients gives
\begin{equation}\label{eq:evenodd}
\widehat{\mathbf 1_B}=\tfrac12\Bigl(\widehat{\mathbf 1_{A'_t}}+\widehat{\mathbf 1_{A'_{-t}}}-2\widehat{\mathbf 1_{A'_t\cap A'_{-t}}}\Bigr)+\tfrac12F*\bigl((e_t-e_{-t})F\bigr),
\end{equation}
the first bracket being $\widehat{\mathbf 1_B}+\widehat{\mathbf 1_{-B}}$ and the last term $\widehat{\mathbf 1_B}-\widehat{\mathbf 1_{-B}}$ (its $m$th coefficient is $a_m(a_{m-t}-a_{m+t})$), and also $\widehat{\mathbf 1_{A'_t}}=F*(e_tF)$ and $\widehat{\mathbf 1_{A'_t\cap A'_{-t}}}=\widehat{\mathbf 1_{A'_t}}*(e_{-t}F)$. We decompose each term.

\emph{The sets $A'_{\pm t}$.} We have $F*(e_tF)=T_1*(e_tT_1)+T_1*(e_tT_2)+T_2*(e_tT_1)+T_2*(e_tT_2)$. The first term has modulus at most $T_1*T_1\le u*u$, so \eqref{eq:split} writes it as $R_1'+V$ with $|R_1'|\le u$ and $\|V\|_2\le K\|T_1*(e_tT_1)\|_1^{1/2}\le K\|T_1\|_1=K\mu/2$. Each of the other three terms has $L^2$ norm at most $\|T_1\|_1\|T_2\|_2$ or $\|T_2\|_1\|T_2\|_2$ by Young's inequality, that is, at most $\frac\mu2(\frac{K\mu}2)^{1/2}=\frac1{2\sqrt2}K^{1/2}\mu^{3/2}\le\frac12K\mu$, where we used $\mu\le2K$. Hence $\widehat{\mathbf 1_{A'_t}}=R_1'+R_2'$ with
\[
|R_1'|\le u,\qquad\|R_1'\|_1\le K,\qquad\|R_2'\|_2\le2K\mu,
\]
and in the same way $\widehat{\mathbf 1_{A'_{-t}}}=R_1''+R_2''$ with the same bounds.

\emph{The set $A'_t\cap A'_{-t}$.} We have $\widehat{\mathbf 1_{A'_t\cap A'_{-t}}}=R_1'*(e_{-t}T_1)+R_1'*(e_{-t}T_2)+R_2'*(e_{-t}F)$. The first term has modulus at most $u*u$ and \eqref{eq:split} writes it as $R_1'''+v$ with $|R_1'''|\le u$ and $\|v\|_2\le K\bigl(\|R_1'\|_1\|T_1\|_1\bigr)^{1/2}\le K^{3/2}(\mu/2)^{1/2}$; the second has $L^2$ norm at most $\|R_1'\|_1\|T_2\|_2\le K^{3/2}(\mu/2)^{1/2}$; the third has Fourier coefficients $\widehat{R_2'}(m)a_{m+t}$, so its $L^2$ norm is at most $\|R_2'\|_2\le2K\mu$. Thus $\widehat{\mathbf 1_{A'_t\cap A'_{-t}}}=R_1'''+R_2'''$ with
\[
|R_1'''|\le u,\qquad\|R_2'''\|_2\le\sqrt2\,K^{3/2}\mu^{1/2}+2K\mu .
\]

\emph{The odd part.} Write $(e_t-e_{-t})F=S_1'+S_2'$ with $S_1'=(e_t-e_{-t})T_1$ and $S_2'=(e_t-e_{-t})T_2$; then $|S_1'|\le2T_1$, $\|S_1'\|_1\le\mu$, $\|S_2'\|_1\le\mu$ and $\|S_2'\|_2\le2\|T_2\|_2\le(2K\mu)^{1/2}$. Now $F*(S_1'+S_2')=T_1*S_1'+T_1*S_2'+T_2*S_1'+T_2*S_2'$. The function $\frac12T_1*S_1'$ has modulus at most $T_1*T_1\le u*u$, so by \eqref{eq:split} the first term equals $S_1+V_2$ with $|S_1|\le2u$ and $\|V_2\|_2\le2K\bigl(\tfrac12\|T_1\|_1\|S_1'\|_1\bigr)^{1/2}\le K\mu$. The other three terms have $L^2$ norms at most $\|T_1\|_1\|S_2'\|_2\le\frac\mu2(2K\mu)^{1/2}$, $\|S_1'\|_1\|T_2\|_2\le\mu(K\mu/2)^{1/2}$ and $\|S_2'\|_1\|T_2\|_2\le\mu(K\mu/2)^{1/2}$, each equal to $K^{1/2}\mu^{3/2}/\sqrt2\le K\mu$. Hence $F*\bigl((e_t-e_{-t})F\bigr)=S_1+S_2$ with
\[
|S_1|\le2u,\qquad\|S_2\|_2\le4K\mu .
\]

\emph{Conclusion.} Put $Q_1^0=\frac12(R_1'+R_1''-2R_1''')+\frac12S_1$ and $Q_2^0=\frac12(R_2'+R_2''-2R_2''')+\frac12S_2$. By \eqref{eq:evenodd}, $\widehat{\mathbf 1_B}=Q_1^0+Q_2^0$, $|Q_1^0|\le\frac12(u+u+2u)+u=3u$, and
\[
\|Q_2^0\|_2\le\tfrac12\bigl(2K\mu+2K\mu+2\sqrt2\,K^{3/2}\mu^{1/2}+4K\mu\bigr)+2K\mu=6K\mu+\sqrt2\,K^{3/2}\mu^{1/2}.
\]
Finally, $\widehat{\mathbf 1_B}(-y)=\overline{\widehat{\mathbf 1_B}(y)}$ and $u$ is even, so $Q_j(y):=\frac12\bigl(Q_j^0(y)+\overline{Q_j^0(-y)}\bigr)$ ($j=1,2$) satisfy $Q_1+Q_2=\widehat{\mathbf 1_B}$, $|Q_1|\le3u$, $Q_1(-y)=\overline{Q_1(y)}$ and $\|Q_2\|_2\le\|Q_2^0\|_2$. The last inequality of the lemma is $6K\mu\le6\sqrt2K^{3/2}\mu^{1/2}$, that is, $\mu\le2K$.
\end{proof}

The third lemma concerns Bedert's twisted function
\[
f_t(y)=2\bigl(1+\sin2\pi ty\bigr)F(y).
\]
Since $(1+\sin2\pi ty)(F+K)\ge0$ and $0\le1+\sin2\pi ty\le2$, we have $f_t\ge-4K$ and $\|f_t\|_1\le4\mu$; and since $2\sin2\pi ty=-i(e_t-e_{-t})$,
\begin{equation}\label{eq:ftcoeff}
\widehat{f_t}(m)=2a_m-ia_{m-t}+ia_{m+t}\in\{0,\ 2,\ 2\pm i,\ \pm i\},
\end{equation}
which vanishes at $m=0$ by symmetry, so $\int f_t=0$, and equals $2-i$ for $m\in B_t(A')$. By Lemma~\ref{lem:conv}, $f_t*f_t\ge-\frac12(4K\cdot4\mu+4K\cdot4\mu)=-16K\mu$. Put
\[
r_t(y)=\bigl(1-\cos(2\pi ty+\pi/4)\bigr)(f_t*f_t)(y)\ \ge\ -32K\mu,
\]
and let $\psi_t(y)=\frac12(r_t(y)+r_t(-y))$ and $\varphi_t(y)=\frac12(r_t(y)-r_t(-y))$ be its even and odd parts. Since $\psi_t\pm\varphi_t\ge-32K\mu$,
\begin{equation}\label{eq:amp}
|\varphi_t(y)|\le\psi_t(y)+32K\mu\qquad(y\in\mathbb T).
\end{equation}
Write $\rho_m=\widehat{r_t}(m)$. As $r_t$ is real, $\rho_{-m}=\overline{\rho_m}$, so $\widehat{\psi_t}(m)=\operatorname{Re}\rho_m$ and $\widehat{\varphi_t}(m)=i\operatorname{Im}\rho_m$. Since $\cos(2\pi ty+\pi/4)=\frac12(e^{i\pi/4}e_t+e^{-i\pi/4}e_{-t})$ and $(f_t*f_t)^\wedge=\widehat{f_t}^{2}$,
\begin{equation}\label{eq:rho}
\rho_m=\widehat{f_t}(m)^2-\frac{1+i}{2\sqrt2}\widehat{f_t}(m-t)^2-\frac{1-i}{2\sqrt2}\widehat{f_t}(m+t)^2 .
\end{equation}

\begin{lemma}[Coefficient bounds]\label{lem:coeff}
For all $m\in\Z$, $\operatorname{Re}\rho_m\le4+\frac1{\sqrt2}$ and $|\operatorname{Im}\rho_m|\le9$. For $m\in B_t(A')$, $-\operatorname{Im}\rho_m\ge4+\sqrt2$.
\end{lemma}

\begin{proof}
By \eqref{eq:ftcoeff} each square in \eqref{eq:rho} lies in $Z=\{0,4,3+4i,3-4i,-1\}$. For $z\in Z$ one has $\operatorname{Re}z\le4$, $-\operatorname{Re}\frac{(1+i)z}{2\sqrt2}=\frac{\operatorname{Im}z-\operatorname{Re}z}{2\sqrt2}\le\frac1{2\sqrt2}$ and $-\operatorname{Re}\frac{(1-i)z}{2\sqrt2}=-\frac{\operatorname{Re}z+\operatorname{Im}z}{2\sqrt2}\le\frac1{2\sqrt2}$, which gives the first bound; and $|\operatorname{Im}z|\le4$, $\bigl|\operatorname{Im}\frac{(1\pm i)z}{2\sqrt2}\bigr|\le\frac7{2\sqrt2}$, which gives $|\operatorname{Im}\rho_m|\le4+\frac7{\sqrt2}<9$. If $m\in B_t(A')$ then $a_m=a_{m-t}=1$ and $a_{m+t}=0$, so $\widehat{f_t}(m)^2=(2-i)^2=3-4i$, $\widehat{f_t}(m-t)=2+i(1-a_{m-2t})\in\{2,2+i\}$ and $\widehat{f_t}(m+t)=-i(1-a_{m+2t})\in\{-i,0\}$. Substituting the four possibilities into \eqref{eq:rho} gives
\[
-\operatorname{Im}\rho_m\in\Bigl\{4+\sqrt2,\ 4+\tfrac5{2\sqrt2},\ 4+\tfrac7{2\sqrt2},\ 4+2\sqrt2\Bigr\},
\]
each of which is at least $4+\sqrt2$.
\end{proof}

\begin{proof}[Proof of Proposition~\ref{prop:boundary}]
Fix $t\ne0$, write $B=B_t(A')$, take $Q_1,Q_2$ from Lemma~\ref{lem:decomp}, and put $E=\|Q_2\|_2$ and $H=|Q_1|$. Then $H\ge0$ is even, so $\widehat H(m)$ is real and $\widehat{H*H}(m)=\widehat H(m)^2\ge0$; moreover $\|H\|_1\le3\int u=3K$ and $\|H\|_2\le\|\widehat{\mathbf 1_B}\|_2+E=|B|^{1/2}+E$. We integrate \eqref{eq:amp} against $H*H\ge0$. On the one hand, by Parseval and Lemma~\ref{lem:coeff},
\begin{align*}
X:={}&\int(\psi_t+32K\mu)(H*H)=\sum_m\operatorname{Re}\rho_m\widehat H(m)^2+32K\mu\|H\|_1^2\\
\le{}&\Bigl(4+\frac1{\sqrt2}\Bigr)\bigl(|B|^{1/2}+E\bigr)^2+288K^3\mu .
\end{align*}
On the other hand $H*H\ge|Q_1*Q_1|$ pointwise, so by \eqref{eq:amp}
\[
X\ge\int|\varphi_t|\,|Q_1*Q_1|\ge\Bigl|\int\varphi_t(Q_1*Q_1)\Bigr|=\Bigl|\sum_mi\operatorname{Im}\rho_m\widehat{Q_1*Q_1}(-m)\Bigr| .
\]
Since $Q_1=\widehat{\mathbf 1_B}-Q_2$ and $\widehat{\mathbf 1_B}*\widehat{\mathbf 1_B}=\widehat{\mathbf 1_B}$, the coefficients of $Q_1*Q_1$ are $\mathbf 1_B(m)-2\cdot\mathbf 1_B(m)\widehat{Q_2}(m)+\widehat{Q_2}(m)^2$. The contribution of the first of these terms is $i\sum_{m\in-B}\operatorname{Im}\rho_m=i\sum_{m\in B}(-\operatorname{Im}\rho_m)$, of modulus at least $(4+\sqrt2)|B|$ by Lemma~\ref{lem:coeff}; by the same lemma, Cauchy--Schwarz and Parseval, the other two contribute at most $9\bigl(2|B|^{1/2}E+E^2\bigr)$ in modulus. Comparing the two bounds for $X$ and using $(4+\sqrt2)-(4+\frac1{\sqrt2})=\frac1{\sqrt2}$,
\[
\frac{|B|}{\sqrt2}\le\Bigl(13+\frac1{\sqrt2}\Bigr)\bigl(2|B|^{1/2}E+E^2\bigr)+288K^3\mu .
\]
With $\alpha=13\sqrt2+1<19.4$ and $\beta=288\sqrt2\,K^3\mu<408K^3\mu$ this reads $|B|\le2\alpha E|B|^{1/2}+\alpha E^2+\beta$, so $|B|^{1/2}\le\alpha E+(\alpha^2E^2+\alpha E^2+\beta)^{1/2}$ and $|B|\le(4\alpha^2+2\alpha)E^2+2\beta$. Since $E^2\le98K^3\mu$ by Lemma~\ref{lem:decomp}, we get $|B|\le\bigl(98(4\alpha^2+2\alpha)+816\bigr)K^3\mu<1.6\cdot10^5K^3\mu$. The second inequality of the proposition is $\mu\le2K$.
\end{proof}

\begin{remark}\label{rem:threeK}
The proof uses only three pieces of information about $F$: the hypothesis $F\ge-K$, the identity $\int u=K$, and $\mu=\|F\|_1$ (through $\|T_1\|_1=\|T_2\|_1=\mu/2$ and $\|T_2\|_2^2\le K\mu/2$). The three factors of $K$ in the final bound enter as follows: one through $f_t\ge-4K$, which produces the constant $32K\mu$ in \eqref{eq:amp}, and two through $\|H\|_1\le3\int u$; the terms in $E^2$ are of the same size. Consequently, within this argument, an improvement to $|B_t(A')|\ll K^{4-\delta}$ can only come from an upper bound $\mu\ll K^{1-\delta}$. Such a bound fails for the sets of Theorem~\ref{thm:base}: for them $\mu/K=\|f_A\|_1/\lambda_A$ lies between $0.59$ and $0.76$ at the sizes of Table~\ref{tab:num} and should tend to $0.97$ along the construction (Section~\ref{sec:numerics}), and we do not know whether such a bound holds for any set with $\lambda_A\le n^{1/2-\varepsilon}$.
\end{remark}

\subsection{Sets with frequencies $O(n)$}\label{subsec:dense}

\begin{lemma}[Total boundary]\label{lem:boundary}
Let $T_3(A'):=\#\{(a,s)\in A'\times\Z_{>0}: a-s\in A',\ a+s\in A'\}$. Then $T_3(A')\le N^2/4$ and
\[
\sum_{t\ne0}|B_t(A')|=N^2-N-2T_3(A')\ \ge\ \frac{N^2}2-N .
\]
\end{lemma}

\begin{proof}
Write $r(t)=\#\{a\in A':a-t\in A'\}$ and $r_3(t)=\#\{a\in A':a-t\in A',\ a+t\in A'\}$, so that $|B_t(A')|=r(t)-r_3(t)$. Summing over $t\ne0$, $\sum_{t\ne0}r(t)=\#\{(a,a')\in A'^2:a\ne a'\}=N^2-N$ and $\sum_{t\ne0}r_3(t)=\#\{(a,t):t\ne0,\ a\pm t\in A'\}=2T_3(A')$, which gives the identity. For the bound on $T_3$, order $A'=\{a_1<\dots<a_N\}$. If $(a_i,s)$ is counted by $T_3(A')$ then $a_i+s=a_j$ for some $j>i$ and $a_i-s=2a_i-a_j=a_l$ for some $l<i$; since $s$ determines $j$ and $l$, the number of such $s$ is at most $\min(N-i,i-1)$, and $\sum_{i=1}^N\min(N-i,i-1)\le N^2/4$.
\end{proof}

\begin{proposition}[Frequencies $O(n)$]\label{prop:dense}
Let $C\ge1$ and let $A\subset[1,Cn]$ be a set of $n\ge2$ positive integers. Then
\[
\lambda_A^3\|f_A\|_1\ \ge\ \frac{n-1}{32CC_1}\qquad\text{and hence}\qquad\lambda_A\ \ge\ \Bigl(\frac{n-1}{64CC_1}\Bigr)^{1/4},
\]
with $C_1$ the constant of Proposition~\ref{prop:boundary}.
\end{proposition}

\begin{proof}
Put $A'=A\cup(-A)$, so $N=2n$, $K=2\lambda_A$ and $\mu=2\|f_A\|_1$. Proposition~\ref{prop:boundary} gives $|B_t(A')|\le16C_1\lambda_A^3\|f_A\|_1$ for every $t\ne0$. If $B_t(A')\ne\emptyset$ then $t\in(A'-A')\setminus\{0\}\subset[-2Cn,2Cn]\setminus\{0\}$, so at most $4Cn$ values of $t$ contribute to the sum in Lemma~\ref{lem:boundary}. Hence
\[
2n^2-2n=\frac{N^2}2-N\ \le\ \sum_{t\ne0}|B_t(A')|\ \le\ 4Cn\cdot16C_1\lambda_A^3\|f_A\|_1,
\]
which is the first claim; the second follows from $\|f_A\|_1\le2\lambda_A$.
\end{proof}

\begin{remark}\label{rem:dense}
(i) The proof shows more generally that every $n$-element set $A\subset\N$ satisfies $\max A\ge(n^2-n)/(32C_1\lambda_A^3\|f_A\|_1)\ge(n^2-n)/(64C_1\lambda_A^4)$. A set for which the exponent $1/5$ of \cite{Bed25,Sha26} is attained, that is $\lambda_A\le n^{1/5+\varepsilon}$, would therefore need $\max A\ge cn^{6/5-4\varepsilon}$: the sets that could be extremal for Chowla's problem in the sense of having a lower bound worse than $n^{1/4}$ are sparse in their hull, whereas the sets constructed in this paper are dense in theirs. (ii) The selection of $t$ in \cite{Bed25} (Lemma~7.4 there) and in \cite{Sha26} produces a shift with $|B_t(A')|\gg N/K$, which is what limits the general exponent to $1/5$; for sets with frequencies $O(n)$, Lemma~\ref{lem:boundary} produces a shift with $|B_t(A')|\gg N$ instead, and the exponent $1/4$ results. (iii) For the sets $A(S)$ of Theorem~\ref{thm:base} the bound is far from tight, since $-\min f_{A(S)}=\frac12\sqrt n+O(n^{1/4}\sqrt{\log n})$; Proposition~\ref{prop:dense} says nothing about whether $\sqrt n$ is the true order for dense sets.
\end{remark}

\subsection{The $L^1$ norm}\label{subsec:L1}

\begin{lemma}[Third moment of the positive part]\label{lem:third}
Let $S_3(A')$ be the number of triples $(a,b,c)\in A'^3$ with $a+b+c=0$. Then
\[
S_3(A')\ \ge\ \frac{\bigl(N-K\mu/2\bigr)_+^2}{\mu/2}-\frac{K^2\mu}2 ,
\]
and if $N\ge2K\mu$ then $S_3(A')\ge N^2/\mu$.
\end{lemma}

\begin{proof}
Expanding $F^3$ shows $S_3(A')=\int F^3=\int F_+^3-\int F_-^3$. By Cauchy--Schwarz,
\[
\int F_+^2=\int F_+^{1/2}F_+^{3/2}\le\Bigl(\int F_+\Bigr)^{1/2}\Bigl(\int F_+^3\Bigr)^{1/2},
\]
so $\int F_+^3\ge(\int F_+^2)^2/\int F_+$, where $\int F_+=\mu/2$ and $\int F_+^2=N-\int F_-^2\ge(N-K\mu/2)_+$ by \eqref{eq:L1trivial} and $\int F_+^2\ge0$; also $\int F_-^3\le K^2\mu/2$ by \eqref{eq:L1trivial}. If $N\ge2K\mu$ then $K\mu/2\le N/4$, so the first term is at least $2(3N/4)^2/\mu=9N^2/(8\mu)$, while $K^2\mu/2=(K\mu)^2/(2\mu)\le N^2/(8\mu)$; the difference is at least $N^2/\mu$.
\end{proof}

The lemma uses only $\int F=0$, $\int F^2=N$ and $F\ge-K$; the integrality of the coefficients enters through the interpretation of $\int F^3$ as a count. Since $\mu\le2K$, it is never weaker than Roth's bound $S_3(A')\ge N^2/K-KN$ (\cite[Lemma~5]{Rot73}, \cite[Lemma~4.1]{Bed25}) by more than a constant factor, and it is stronger by the factor $K/\mu$. We combine it with the total-boundary identity of Shankar \cite[Lemma~3.1]{Sha26}, whose short proof we include.

\begin{lemma}[Shankar]\label{lem:shankar}
For every finite symmetric $A'\subset\Z\setminus\{0\}$,
\[
\sum_{t\in A'}|B_t(A')|\ \ge\ \frac13S_3(A') .
\]
\end{lemma}

\begin{proof}
Let $P=A'\cap\N$, $Q=\#\{(u,v)\in P^2:u+v\in P\}$ and $I=\#\{(u,v)\in P^2:u-v\in A',\ u+v\in A'\}$. A solution of $a+b+c=0$ in $A'^3$ has no zero entry and not all entries of one sign, so it has exactly two positive entries or exactly two negative entries. In the first case the two positive entries form, in order of position, a pair $(u,v)\in P^2$ with $-(u+v)\in A'$, that is $u+v\in P$, and the position of the negative entry can be chosen in three ways; so there are $3Q$ such solutions, and by symmetry $3Q$ of the second kind: $S_3(A')=6Q$. Next, with $r$ and $r_3$ as in the proof of Lemma~\ref{lem:boundary}, $\sum_{t\in A'}r(t)=\#\{(a,t)\in A'^2:a-t\in A'\}=S_3(A')$, via the bijection $(a,t)\mapsto(a,-t,t-a)$, and $\sum_{t\in A'}r_3(t)=\#\{(a,t)\in A'^2:a-t,\ a+t\in A'\}=4I$, because $(a,t)\mapsto(|a|,|t|)$ maps the pairs counted on the left onto the pairs counted by $I$ (as $\{|a|-|t|,|a|+|t|\}\subset\{\pm(a-t),\pm(a+t)\}$ and $A'=-A'$) with fibres of size four. Hence $\sum_{t\in A'}|B_t(A')|=6Q-4I\ge2Q$, since $I\le\#\{(u,v)\in P^2:u+v\in A'\}=Q$.
\end{proof}

\begin{theorem}[An $L^1$-sensitive lower bound]\label{thm:L1}
For every finite $A\subset\N$ with $|A|=n$,
\[
\lambda_A^3\|f_A\|_1^2\ \ge\ \frac{n}{48C_1},
\]
with $C_1$ the constant of Proposition~\ref{prop:boundary}. Consequently $\lambda_A\ge(n/(192C_1))^{1/5}$, and if $\|f_A\|_1\le\lambda_A^{1-\delta}$ for some $\delta\in(0,1]$ then $\lambda_A\ge(n/(48C_1))^{1/(5-2\delta)}$.
\end{theorem}

\begin{proof}
Let $A'=A\cup(-A)$, so $N=2n$, $K=2\lambda_A\ge1$ and $\mu=2\|f_A\|_1$. If $N<2K\mu$, that is $n<4\lambda_A\|f_A\|_1$, then $\lambda_A^3\|f_A\|_1^2=\lambda_A(\lambda_A\|f_A\|_1)^2>\lambda_An^2/16\ge n/32\ge n/(48C_1)$, using $\lambda_A\ge1/2$. Otherwise Lemmas~\ref{lem:third} and~\ref{lem:shankar} and Proposition~\ref{prop:boundary} give
\[
\frac{2n^2}{3\|f_A\|_1}=\frac{N^2}{3\mu}\ \le\ \sum_{t\in A'}|B_t(A')|\ \le\ 2n\cdot C_1(2\lambda_A)^3\cdot2\|f_A\|_1=32C_1n\lambda_A^3\|f_A\|_1 ,
\]
which is the claim. The consequences follow from $\|f_A\|_1\le2\lambda_A$ and from $\|f_A\|_1\le\lambda_A^{1-\delta}$ respectively.
\end{proof}

Theorem~\ref{thm:L1} says that a set for which the exponent $1/5$ is close to sharp must have a large $L^1$ norm: if $\lambda_A\le n^{1/5+\varepsilon}$ then $\|f_A\|_1\ge n^{1/5-3\varepsilon/2}/(48C_1)^{1/2}$, so its cosine sum is, in $L^1$, as large as its negative excursion up to a factor $n^{5\varepsilon/2}$. The known near-extremal sets for the upper bound behave in the same way: for the sets of Theorem~\ref{thm:base}, $\|f_A\|_1=(0.48+o(1))\sqrt n$ numerically (Section~\ref{sec:numerics}), so $\|f_A\|_1\approx\lambda_A$ and the theorem gives nothing beyond $1/5$ for them. Whether every set with $\lambda_A\le n^{1/2-\varepsilon}$ has $\|f_A\|_1\asymp\lambda_A$ is a natural question that we cannot answer; a positive answer would not by itself improve the exponent, since Theorem~\ref{thm:L1} gains only from $\|f_A\|_1\ll\lambda_A$.

The same substitution of $\int F_+$ for $\int(F+K)$ in a Cauchy--Schwarz step gives an $L^1$ form of Ruzsa's lemma \cite[Lemma~3.1]{Ruz04}, in the formulation of \cite[Lemma~4.2]{Bed25}.

\begin{lemma}[Ruzsa's lemma with the $L^1$ norm]\label{lem:ruzsaL1}
Let $U,V\subset\Z\setminus\{0\}$ be finite with $|U|=|V|$, $U\subset A'$, $V\cap A'=\emptyset$ and $U-V\subset A'$. Then $|U|\le4K\mu$; for $A'=A\cup(-A)$ this is $|U|\le16\lambda_A\|f_A\|_1$.
\end{lemma}

\begin{proof}
Let $g=\widehat{\mathbf 1_U}-\widehat{\mathbf 1_V}$. By Parseval, $\int gF=\#(U\cap A')-\#(V\cap A')=|U|$. Split $\int gF=\int gF_+-\int gF_-$. Cauchy--Schwarz with the weight $F_+\ge0$ gives $|\int gF_+|\le(\int F_+)^{1/2}(\int|g|^2F_+)^{1/2}$, and $\int|g|^2F_+=\int|g|^2F+\int|g|^2F_-$. Here $\int|g|^2F=\sum_{x,y\in U\cup V}\varepsilon_x\varepsilon_y\mathbf 1_{A'}(x-y)$ with $\varepsilon=+1$ on $U$ and $-1$ on $V$; the diagonal vanishes since $0\notin A'$, the $U\times U$ and $V\times V$ terms contribute at most $|U|(|U|-1)+|V|(|V|-1)$, and each of the $2|U||V|$ cross terms equals $-1$ because $U-V\subset A'=-A'$. So $\int|g|^2F\le-2|U|\le0$, and $\int|g|^2F_-\le K\int|g|^2=2K|U|$. Hence $|\int gF_+|\le(\tfrac\mu2\cdot2K|U|)^{1/2}=(K\mu|U|)^{1/2}$. Also $|\int gF_-|\le\|g\|_2\|F_-\|_2\le(2|U|)^{1/2}(K\mu/2)^{1/2}=(K\mu|U|)^{1/2}$. Therefore $|U|\le2(K\mu|U|)^{1/2}$, that is, $|U|\le4K\mu$.
\end{proof}

\begin{corollary}\label{cor:APL1}
Every arithmetic progression contained in $A$ has at most $64\lambda_A\|f_A\|_1+6$ terms.
\end{corollary}

\begin{proof}
Let $\{y_0+jd:0\le j<L\}\subset A$ with $d\ne0$ and $L\ge3$, put $p=\lfloor(L-1)/2\rfloor\ge1$, $x_0=y_0+pd$ and $P=\{x_0-pd,\dots,x_0+pd\}\subset A\subset A'$; then $-P\subset A'$ as well. We follow Ruzsa \cite{Ruz04} in the form given in \cite[Lemma~4.4]{Bed25}. Let $U_0=\{x_0+d,\dots,x_0+pd\}$ and $U_\ell=U_0+\ell d$ for $\ell\ge0$. Every element of $U_0-(U_0-x_0)+\{0,d\}$ is of the form $x_0+(i-j+\epsilon)d$ with $1\le i,j\le p$ and $\epsilon\in\{0,1\}$, so $U_0-(U_0-x_0)+\{0,d\}\subset P$. Since $|U_0\cap A'|=p$ and $U_\ell\cap A'=\emptyset$ for large $\ell$, there is $\ell\ge0$ with $|U_\ell\cap A'|\ge p/2>|U_{\ell+1}\cap A'|$. If $|(U_\ell-x_0)\cap A'|<p/2$, put $U=U_\ell\cap A'$ and $V=(U_\ell-x_0)\setminus A'$; then $U-V\subset U_0-(U_0-x_0)\subset P$. Otherwise put $U=(U_\ell-x_0)\cap A'$ and $V=U_{\ell+1}\setminus A'$; then $U-V\subset(U_0-x_0)-U_0-d=-\bigl(U_0-(U_0-x_0)+d\bigr)\subset-P$. In both cases $U\subset A'$, $V\cap A'=\emptyset$, $U-V\subset A'$, $|U|\ge p/2$ and $|V|>p/2$. Removing $0$ from $V$ if it occurs there and discarding elements until $|U|=|V|$, we obtain sets as in Lemma~\ref{lem:ruzsaL1} with $|U|=|V|\ge p/2-1$. Hence $p/2-1\le16\lambda_A\|f_A\|_1$ and $L\le2p+2\le64\lambda_A\|f_A\|_1+6$.
\end{proof}

Ruzsa's bound $O(\lambda_A^2)$ is recovered from $\|f_A\|_1\le2\lambda_A$. For an interval $A=[1,M]$ one has $\lambda_A\asymp M$ and $\|f_A\|_1\asymp\log M$, so Corollary~\ref{cor:APL1} is sharp up to the logarithm for the interval, whereas Ruzsa's bound is off by a factor of $M$ there.

\section{Numerical verification}\label{sec:numerics}

We implemented the construction of Sections~\ref{sec:singer} to~\ref{sec:thinning} in Python (NumPy). The Singer set is obtained as the set of exponents $i\in\{0,\dots,v-1\}$ for which $\vartheta^i$ lies in a fixed two-dimensional subspace of $\mathbb F_{q^3}$ over $\mathbb F_q$, where $\vartheta$ is a primitive element of $\mathbb F_{q^3}$, and the perfect-difference property is verified directly. The tent thinning with parity block is sampled by rejection until $|S|=k$ exactly. At the sizes tested, the parameter $r$ of Section~\ref{sec:thinning} exceeds $q+1$ (it is designed for the asymptotic proof only), so the parity block was taken to be the $r=\lceil3\sqrt q\rceil$ odd-offset elements closest to the centre; the seed of the random generator was $0$ throughout. The minimum of $f_A$ was located by a real FFT of the indicator of $A$ on a grid of at least $16\max A$ points, followed by golden-section refinement of the $64$ lowest grid values using the identity of Proposition~\ref{prop:identity}. The identity itself was checked against the definition of $f_A$ on a coarse grid, with agreement to $10^{-9}$ in all cases, and the absence of coincident frequencies was verified for every set. Table~\ref{tab:num} lists the results.

\begin{table}[ht]
\centering
\small
\setlength{\tabcolsep}{4pt}
\begin{tabular}{rrrrrrrrr}
\hline
$q$ & $k$ & $\theta$ & $r$ & $n$ & $\max A$ & $\min_xf_A$ & $\dfrac{-\min f_A-k/2}{\sqrt k}$ & $\dfrac{-\min f_A}{\sqrt n}$\\
\hline
61 & 12 & 0.749 & 24 & 132 & 7845 & $-9.639$ & 1.051 & 0.839\\
101 & 20 & 0.665 & 31 & 380 & 26853 & $-15.689$ & 1.272 & 0.805\\
199 & 40 & 0.635 & 43 & 1560 & 106144 & $-28.646$ & 1.367 & 0.725\\
311 & 60 & 0.505 & 53 & 3540 & 265383 & $-42.141$ & 1.567 & 0.708\\
401 & 80 & 0.522 & 61 & 6320 & 418031 & $-53.550$ & 1.515 & 0.674\\
601 & 120 & 0.525 & 74 & 14280 & 1025874 & $-78.215$ & 1.663 & 0.655\\
809 & 160 & 0.496 & 86 & 25440 & 1753844 & $-101.213$ & 1.677 & 0.635\\
\hline
\end{tabular}
\medskip
\caption{Numerical minima of $f_{A(S)}$, $n=|A(S)|=k(k-1)$, for the construction with parity block $r=\lceil3\sqrt q\rceil$ and random seed $0$. The last column decreases towards $1/2$; the penultimate column grows slowly, consistent with an error term of order $\sqrt{k\log k}$.}
\label{tab:num}
\end{table}

Two further observations illustrate the mechanism. First, in every row of Table~\ref{tab:num} the minimiser satisfies $\rho^2<1$, so the excess over $k/2$ is entirely the value of $g/2$ at a zero of $\rho$, as in Remark~\ref{rem:sharp}. Second, at $q=401$, $k=80$ and seed $0$, removing the parity block ($r=0$) gives $\min f_A=-55.93$, attained at $x=0.3376\pi$, adjacent to the point $\pi/3$ of Zone II, with $\rho(x)^2<10^{-3}$ and $g(x)/2=15.9$; with $r=20,61,120$ the minimum of this realisation is $-54.25$, $-53.55$, $-52.67$ and moves away from $\pi/3$. This single realisation should not be read as a measurable improvement from the block at these sizes: changing $r$ also changes $\theta$ and hence redraws the whole random set, and over twelve seeds at the same $q$, $k$ the minima range from $-58.4$ to $-52.0$ with $r=0$ (mean $-54.0$) and from $-56.5$ to $-52.3$ with $r=61$ (mean $-54.2$), a seed-to-seed spread of order $\sqrt k$ that is as large as the difference above. The parity block is needed for the proof of Zone II, which controls the neighbourhood of $\pi/3$ uniformly rather than for a typical realisation; at the sizes tested its effect is within the random fluctuation.

The $L^1$ norm $\|f_A\|_1=(2\pi)^{-1}\int_0^{2\pi}|f_A|$ of the same sets, computed as the mean of $|f_A|$ over the FFT grid, is $0.496$, $0.490$, $0.486$, $0.484$, $0.484$, $0.483$, $0.482$ times $\sqrt n$ for the seven rows of Table~\ref{tab:num}. This is consistent with the identity of Proposition~\ref{prop:identity}: away from the special points $\rho$ behaves like a centred Gaussian with variance $k/2$, so $\rho^2-k/2$ behaves like $\frac k2(Z^2-1)$ with $Z$ standard normal, and $\E|Z^2-1|=4\varphi(1)=0.968$ gives $\|f_A\|_1\approx0.484k\approx0.484\sqrt n$. Thus $\|f_A\|_1/\lambda_A$, which is $0.59$ in the first row and $0.76$ in the last, should tend to $0.97$ along the construction, and Theorem~\ref{thm:L1} yields nothing beyond the exponent $1/5$ for these sets.

\subsection*{Acknowledgements}
An AI language model first suggested the two-sided Sidon square identity of Proposition~\ref{prop:identity}. The authors have checked that identity, together with every other statement and proof in this paper, and take full responsibility for the contents.

\end{document}